\documentclass[12pt]{amsart} 
\usepackage{a4wide}
\usepackage[utf8]{inputenc} 
\usepackage{amsmath}
\usepackage{amssymb}
\usepackage{amstext}

\renewcommand{\phi}{\varphi}

\usepackage{graphicx} 

\usepackage{booktabs} 
\usepackage{array} 
\usepackage{paralist} 
\usepackage{verbatim} 
\usepackage{subfig} 

\usepackage{fancyhdr} 
\newcommand{\ind}[1]{\,^{#1}\kern -1 pt}
\newcommand{\da}{\kern -3 pt \downarrow}

\date{} 

\renewcommand{\a}{{\bf a}}
\renewcommand{\b}{{\bf b}}
\renewcommand{\c}{{\bf c}}
\renewcommand{\d}{{\bf d}}
\renewcommand{\u}{{\bf u}}
\renewcommand{\t}{{\bf t}}
\newcommand{\z}{{\bf z}}
\newcommand{\x}{{\bf x}}
\newcommand{\y}{{\bf y}}

\renewcommand{\r}{\mathbf{r}}
\newcommand{\V}{{\bf v}}

\newcommand{\w}{{\bf w}}
\newcommand{\m}{\mathrm{max}}
\newcommand{\M}{\mathcal{M}}
\newcommand{\fla}{\mbox{FLA}(\Omega)}

\newcommand{\FLA}{\mathrm{FLA}(\Omega)}

\newcommand{\FA}{\mathrm{FAM}(\Omega)}

\newcommand{\FI}{\mathrm{FIM}(\Omega)}
\newcommand{\FG}{\mathrm{FG}(\Omega)}
\newcommand{\R}{\mathbf{R}}

\title{Coherency,  free inverse monoids and free left ample monoids}

\subjclass[2010]{20 M 05, 20 M 30}
\thanks{The  authors acknowledge the support of EPSRC grant no.
EP/I032312/1. 
Research also partially supported by the Hungarian Scientific Research Fund (OTKA) grant no. K83219.}
\date{\today}

\author{Victoria Gould} 
\email{victoria.gould@york.ac.uk}
\author{Mikl\'{o}s Hartmann}
\email{miklos.hartmann@york.ac.uk}
\address{Department of Mathematics\\University
  of York\\Heslington\\York YO10 5DD\\UK}
  
\newtheorem{Lem}{Lemma}[section]

\newtheorem{Thm}[Lem]{Theorem}
\newtheorem{Cor}[Lem]{Corollary}

\newtheorem{cor}[Lem]{Corollary}
\newtheorem{lem}[Lem]{Lemma}
\newtheorem{prop}[Lem]{Proposition}
\theoremstyle{definition}
\newtheorem{Def}[Lem]{Definition}

\begin{document}

\begin{abstract}
A monoid $S$ is {\em right coherent} if every finitely generated subact of every finitely presented right $S$-act is finitely presented. The corresponding notion for a ring $R$ states that every finitely generated submodule of every finitely presented right $R$-module is finitely presented.  For  monoids (and rings) right coherency is a finitary property which determines the existence of a {\em model companion} of the class of right $S$-acts (right $R$-modules) and hence that the class of existentially closed right $S$-acts  (right $R$-modules) is axiomatisable.

Choo, Lam and Luft have shown that free rings are right (and left) coherent; the authors, together with Ru\v{s}kuc, have shown that groups, and free monoids, have the same properties. We demonstrate that free inverse monoids do not.

Any  free inverse monoid contains as a submonoid the free left ample monoid,  and indeed the free monoid, on the same set of generators.   The main objective of the paper is to show that the free left ample monoid {\em is} right coherent. 
Furthermore, by making use of the same techniques we show that both  free inverse and free left ample monoids satisfy $({\bf R})$, $({\bf r)}$, $({\bf L})$ and $({\bf l)}$, conditions arising from the axiomatisability of classes of right $S$-acts and of left $S$-acts.
\end{abstract}

\maketitle
\section{Introduction}\label{sec:intro}

Let $S$ be a monoid. A  {\em right $S$-act} is a set $A$ together with a map $A\times S\rightarrow A$ 
where $(a,s)\mapsto as$, such that
for all $a\in A$ and $s,t\in S$ we have $a1=a$ and $a(st)=(as)t$.  We also have the dual notion of a {\em left $S$-act}: where handedness for $S$-acts is not specified in this article we will always mean {\em right} $S$-acts. The study of  $S$-acts is, effectively, that of representations of the monoid $S$ by mappings of sets.

Clearly $S$-acts over a monoid $S$  are the non-additive 
analogue of $R$-modules over a (unital) ring $R$. Although the study of the two notions diverges considerably once technicalities set in, one can often begin by forming analagous notions and asking analagous questions. In this article we study   coherency for monoids.  A monoid $S$ is said to be {\em right coherent} if every finitely generated subact of every finitely presented right $S$-act is finitely presented. {\em Left coherency} is defined dually; $S$ is {\em coherent} if it is both left and right coherent. These notions are analogous to those for a  ring $R$ (where, of course, $S$-acts are replaced by $R$-modules).
Coherency is a finitary condition for rings and monoids, much weaker than, for example, the condition that says all finitely generated $R$-modules or $S$-acts be finitely presented.
As demonstrated  by Eklof and Sabbagh \cite{eklof:1971}, it is intimately related to the model theory of $R$-modules. The corresponding results for $S$-acts appear in  \cite{gould:1987}, the latter informed by the more general approach of Wheeler
\cite{wheeler:1976}.  

Chase \cite{chase:1960} gave internal conditions on a ring $R$ such that $R$ be right coherent. Correspondingly,  a monoid $S$ is right coherent if and only if for any finitely generated right congruence $\rho$ on $S$, and for any $a,b\in S$, the right congruence \[r(a\rho)=\{ (u,v)\in S\times S:au\,\rho\, av\}\] is finitely generated, and the subact $(a\rho)S\cap (b\rho)S$ of the right $S$-act $S/\rho$ is finitely generated \cite{gould:1992}. 

Choo, Lam and Luft \cite[Corollary 2.2 and remarks]{choo:1972} have shown that free  rings are coherent. The first author proved that free commutative monoids are coherent \cite{gould:1992} and recently the authors, together with Ru\v{s}kuc \cite{ghr:2013}, have shown that free monoids are coherent. 
The class of coherent inverse monoids contains all semilattices of groups \cite{gould:1992} and so, in particular, all groups and all semilattices. Certainly then free groups are coherent. It therefore becomes   natural to ask whether free inverse monoids are coherent, since, not only are they  free objects in a variety of unary algebras,  they are constructed from free groups acting on semilattices. In fact, as we show at the end of this article, coherency fails for free inverse monoids. This negative result motivates us to ask whether  free left ample monoids, which may be thought of as the `positive' part of free inverse monoids, being constructed from free monoids rather than free groups, are coherent. 
We argue that free left ample monoids  are right but not left coherent. The proof of  right coherency  is motivated by the methods in \cite{ghr:2013}, it is, however, rather more delicate. 

For the convenience of the reader we describe in  Section~\ref{sec:prelims}  the construction of the free inverse $\FI$, free left ample $\FLA$ and free ample $\FA$ monoids  on a set $\Omega$ from (prefix) closed subsets of the free  group $\FG$.  In Section~\ref{sec:fiRr} we focus on showing that the finitary properties ($\R$),($\r$),($\mathbf{L}$) and ($\mathbf{l}$) (defined therein) hold for $\FI$ and $\FLA$. These properties (which arise from considerations of first order axiomatisability of the class of strongly flat right and left $S$-acts - see \cite{gould:1987b}) are similar in flavour, although easier to handle, than coherency.
Our main work is in Section~\ref{sec:positive}, where we make a detailed analysis of finitely generated right congruences on  $\FLA$. This hard work is then put to use in Section~\ref{sec:flacoherent} where we show that  $\FLA$ is right coherent for any set $\Omega$. In Section~\ref{sec:constructions} we argue that the class of right coherent monoids is closed under retract. As a consequence of this, we have an alternative (albeit rather longer) proof that free monoids are coherent. Finally, in Section~\ref{sec:negative}, we show that $\FI$, $\FLA$ and $\FA$ are not coherent (for $|\Omega|\geq 2$).

\section{Preliminaries}\label{sec:prelims}

Let $\Omega$ be a non-empty set, let $\Omega^*$ be the free monoid and let $\FG$ be the free group on $\Omega$, 
respectively.
We follow standard practice and denote by $l(a)$ the length of a reduced word $a\in \FG$ and so, in particular, of $a\in\Omega^*$. 
The empty word will be denoted by $\epsilon$.
Of course, $\Omega^*$ is a submonoid of the free group $\mbox{FG}(\Omega)$, and in the sequel, if $a \in \Omega^*$, by $a^{-1}$ we mean the inverse of $a$ in $\mbox{FG}(\Omega)$. For any  $a\in \FG$ we denote by $a\da$ the set of
prefixes of the {\em reduced word} corresponding to $a$. Thus, if $a$ is reduced and $a=x_1\hdots x_n$
where $x_i\in \Omega\cup\Omega^{-1}$, then
\[a\da=\{ \epsilon, x_1,x_1x_2,\hdots, x_1x_2\hdots x_n\}.\]

The free inverse monoid  on $\Omega$ is denoted by $\FI$. The structure of $\FI$ was determined by Munn 
\cite{munn:1974} and Scheiblich \cite{scheiblich:1972}; the description we give below follows that of \cite{scheiblich:1972}, of which further details may be found in \cite{ho}. However, we keep the equivalent characterisation via Munn trees constantly in mind. 

Let $\mathcal{P}^f_c(\Omega)$ be the set of finite prefix closed subsets of $\FG$. If 
$A\in \mathcal{P}^f_c(\Omega)$, then - regarding elements of $A$ as reduced words - a {\em leaf} $a$ of $A$ is a word such that $a$ is not a proper prefix of any other word in $A$.
Note that $\FG$ acts in the obvious way on its semilattice of subsets under union.
Using this action we define 
\[
\FI=\{ (A,a):A\in \mathcal{P}^f_c(\Omega), a\in A\}.
\]
With binary operation given by
\[
(A,a)(B,b)=(A\cup aB, ab),
\]
$\FI$ is the free inverse monoid generated by $\Omega$.
The identity is $(\{ \epsilon\}, \epsilon)$, the inverse $(A,a)^{-1}$ of $(A,a)$ is  $(a^{-1}A,a^{-1})$ and
the natural injection of $\Omega\rightarrow \FI$ is given by
\[x\mapsto (\{ 1,x\},x).\]
We will make use of the fact that the free inverse monoid (in fact, every inverse monoid) possesses a left-right duality, by virtue of the anti-isomorphism given by $x\mapsto x^{-1}$.
For future purposes we remark that if $a\in FG(X)$ is reduced, then
\[a^{-1}\cdot a\,\, {\da}=(a^{-1}){\downarrow}.\]

Throughout this article we denote elements of $\FI$ by boldface letters, 
elements of $\mathcal{P}^f_c(\Omega)$ by capital letters, and elements of $\FG$ by lowercase letters. We write a  typical element of $\FI$ as ${\bf a}=(A,a)$; $A$ and $a$ will always denote the first and second coordinate of
$\a$, respectively. 
One exception to this convention is that we denote the identity $(\{ \epsilon\}, \epsilon)$ of $\FI$ by $\bf{1}$.

The free left ample monoid $\FLA$  on $\Omega$ is the submonoid of $\FI$ given by
\[\FLA=
\{ (A,a)\in \FI: A\subseteq \Omega^*\},\]
note that perforce, $a\in \Omega^*$ and we assume from the outset, when dealing with an element
$\a=(A,a)\in \FLA$, that all the words in $A$ are
reduced. 
We remark that FLA$(\Omega)$ also possesses a unary operation of $(A,a)^+=(A,\epsilon)=(A,a)(A,a)^{-1}$ and 
(as a unary semigroup) is the free algebra on $\Omega$ in both the variety of left restriction semigroups and the quasi-varieties of (weakly) left ample semigroups 
\cite{fountain:1991,gomes:2000,cornock:2011}.

Similarly, the free ample semigroup on $\Omega$ is the submonoid of $\FI$ given by
\[
\FA=\{(A,a) \in \FI: a \in \Omega^*\}.
\]

The free ample monoid possesses another unary operation defined by \[(A,a)^*=(A,a)^{-1} (A,a)=(a^{-1}A,1)\] and (as a biunary semigroup) is the free algebra on $\Omega$ in both the variety of restriction semigroups and the quasi-varieties of (weakly) ample semigroups. We remark here that the set of identities and quasi-identities definining the class of ample monoids is left-right dual, so that $\FA$ consequently also has a left-right duality.

Note that $\FLA$ is built from $\Omega^*$ (see \cite{gould:2009}),but to simplify notation we make use of the embedding of $\Omega^*$ into $\FG$.
However, when dealing with $\FLA$, we will use inverses only when we know that the resulting element lies in $\Omega^*$, for example we will write $u^{-1}v$ only if $u$ is a prefix of $v$.

Let $S$ be a semigroup, let $H \subseteq S \times S$ and let us denote by $\rho$ the right congruence generated by $H$.
Then it is well known that $s \mathrel{\rho} t$ if and only if there exists a so-called $H$-sequence
\[
s=c_1t_1, d_1t_1=c_2t_2, \ldots, d_nt_n=t
\]
connecting $s$ to $t$ where $(c_i,d_i) \in H \cup H^{-1}$ for all $1\leq i\leq n$.
If $n=0$, we interpret this sequence as being $s=t$.

\section{$\FI,\FA$ and $\FLA$ satisfy ($\bf{R}$), ($\bf{r}$), ($\bf{L}$) and ($\bf{l}$).}\label{sec:fiRr}

The conditions $(\bf{R})$ and $(\bf{r})$ \big(($\bf{L}$) and ($\bf{l}$)\big) are connected to the axiomatisability of certain classes of right (left) acts, and were introduced in \cite{gould:1987b}. Connected via axiomatisability to coherency, they are somewhat easier to handle.  
In this section we show that  the free inverse, the free ample and the free left ample monoids satisfy these conditions. In doing so we
develop some facility for handling products and factorisations in these monoids. 

\begin{Def}
Let $S$ be a monoid. We say that $S$ satisfies Condition $(\bf{r})$ if for every $s,t \in S$ the right ideal
\[
{\bf r}^S(s,t)=\{u \in S:su=tu\}
\]
is finitely generated.

The monoid $S$ satisfies Condition $(\bf{R})$ if for every $s,t \in S$ the $S$-subact
\[
{\bf R}^S(s,t)=\{(u,v):su=tv\}
\]
of the right $S$-act $S \times S$ is finitely generated.
(Note that we allow $\emptyset$ to be an ideal and an $S$-act.)

The conditions $(\bf{L})$ and $(\bf{l})$ are defined dually.
\end{Def}


\begin{lem}\label{First}
Let $A$ be a prefix closed subset of $\FG$ and let $g,h \in A$.
Then \[g((g^{-1}h)\da) \subseteq A.\]
\end{lem}

\begin{proof}
Let $x$ be the longest common prefix of the reduced words $g,h \in \FG$.
That is, $g=xg'$ and $h=xh'$ where $g',h'$ do not have a common nonempty prefix.
Then 
\[
g((g^{-1}h)\da)=xg'(g'^{-1} h')\da \subseteq (xg')\da \cup (xh')\da=g\da \cup h\da \subseteq A.
\]
\end{proof}

\begin{lem}\label{lem:crack}
Let $S$ denote either $\FI$, $\FLA$ or $\FA$, let $\a\u=\b\V$ in $S$ and suppose that there exists a leaf $x \in A \cup aU=B \cup bV$ such that $x \not\in A \cup B$.
Then there exist $\u',\V',\z \in S$ such that $\left| A \cup aU'\right|<\left|A \cup aU\right|$,
\[
\a\u'=\b\V' \text{ and } (\u,\V)=(\u',\V')\z.
\]
Furthermore, if $\u=\V$ then $\u'=\V'$.
\end{lem}

\begin{proof}
Clearly $\u\neq \mathbf{1}$. If $S=\FLA$ then it is easy to see that $x=ak$ where $k\in \Omega^*\setminus \{ \epsilon\}$ is a leaf of
$U$. 
The statement for $S$ now follows from Lemma \ref{Crack}. We therefore consider the case where $S=\FI$ of
$S=\FA$. 

We can suppose that the words $x,a,b,u$ and $v$ are reduced.
Note that $x \not\in  A \cup B$ implies that $x\in  aU \cap bV$.
We have that $x \not \in A$ so in particular, $x$ is not a prefix of $a$.
In this case the last letter of $x$ does not cancel in the product $a^{-1}x$.
Now if $a^{-1}x$ is not a leaf of $U$ then there exists $c \in \Omega \cup \Omega^{-1}$, different from the  last letter of $x$, such that $a^{-1}xc \in U$. In this case $xc \in A \cup aU$, contradicting that $x$ is a leaf of $A \cup aU$.
So we have shown that $a^{-1}x$ is a leaf of $U$.
Similarly $b^{-1}x$ is a leaf of $V$.
There are two different cases to consider.

Case (i): $x \neq au$.
Let $z=(au)^{-1}x$.
Note that $u,a^{-1}x \in U$, which is prefix closed, and $z=(au)^{-1}x=u^{-1}\cdot a^{-1}x$.
Lemma \ref{First} then gives that $u(z\da) \subseteq U$.
Since $uz=a^{-1}x$, we have that
\[
(U,u)=(U \setminus \{a^{-1}x\},u)(z\da,1).
\]
Furthermore, $z=(au)^{-1}x=(bv)^{-1}x$, so similarly we have that
\[
(V,v)=(V \setminus \{b^{-1}x\},v)(z\da,1).
\]
Also, $A \cup a(U \setminus \{a^{-1}x\})=B \cup b(V \setminus \{b^{-1}x\})=(A \cup aU) \setminus \{x\}$, so we have that
\[
(A,a)(U\setminus\{a^{-1}x\},u)=(B,b)(V \setminus\{b^{-1}x\},v).
\]
So if we let 
\[
(U',u')=(U \setminus \{a^{-1}x\},u), (V',v')=(V \setminus\{b^{-1}x\},v) \text{ and }\z=(z\da,z),
\]
then (noticing that if $(U,u)=(V,v)$ we must have that $a=b$), the statements of the lemma are satisfied.

Case (ii): $x=au=bv$. 
Since $x \not\in A \cup B$, but $a,b \in A \cup B$ we have that $u,v\neq\epsilon$.
In case $S=\FA$, this implies that the last letters of $x,u$ and $v$ are the same which we denote by $z\in \Omega$.
Note that $uz^{-1},vz^{-1} \in \Omega^*$ in this case.

If $S=\FI$ then let $z$ be the last letter of the reduced word $x$.
If $z$ is not the last letter of $u$ then in the product $x=au$, all letters of $u$ must cancel, so $a=xu^{-1}$ where $xu^{-1}$ is reduced.
However, this contradicts the fact that $x$ is a leaf, showing that the last letter of the reduced word $u$ is $z$.
Similarly the last letter of the reduced word $v$ is $z$.

In both the cases $S=\FA$ and $S=\FI$, $u \neq uz^{-1}$ and $u \neq \epsilon$ imply that $uz^{-1} \in U \setminus \{u\}$, and similarly $vz^{-1} \in V \setminus \{v\}$.
Now let $\u'=(U \setminus\{u\},uz^{-1}), \V'=(V \setminus \{v\},vz^{-1})$ and $\z=(\{1,z\},z)$.
Then 
\[
(U,u)=(U',u')(\{1,z\},z),\, (V,v)=(V',v'),(\{1,z\},z)
\]
and
\[
(A,a)(U',u')=\big((A\cup aU) \setminus \{au\},au'\big)=(B,b)(V',v').
\]
Furthermore, if $\u=\V$ then clearly $\u'=\V'$, which finishes the proof.

\end{proof}

\begin{prop}\label{lem:fi}
The monoids $\FI$, $\FA$ and $\FLA$ satisfy $(\R)$ and $(\r)$.
\end{prop}
\begin{proof}
Let $S$ denote $\FI$, $\FA$ or $\FLA$ and let $\a,\b\in S$. 
We claim that the finite set
\[
X=\{(\u,\V):\a\u=\b\V,\, A\cup aU = A \cup B\}
\]
generates $\mathbf{R}(\a,\b)$.
Let $(\u,\V) \in \R(\a,\b)$.
We prove by induction on the size of $A\cup aU$ that $(\u,\V) \in X \cdot S$.
Note that $A \cup aU=B \cup bV$ implies $A \cup B \subseteq A \cup aU$, so that if $\left| A \cup aU\right| \leq \left|A \cup B\right|$, then necessarily $A \cup aU=B \cup bV=A \cup B$, which shows that $(\u,\V) \in X$.

Suppose now that we have that there exists an $n \geq \left|A \cup B\right|$ such that whenever $\left|A \cup aU\right|\leq n$ and $(\u,\V) \in \R(\a,\b)$, then necessarily $(\u,\V) \in X \cdot S$.
Now let $(\u,\V) \in \R(\a,\b)$ be such that $\left| A \cup aU\right|=n+1$.
Since $(\u,\V) \in \R(\a,\b)$ we have that $A \cup B \subseteq A \cup aU=B \cup bV$, and since $n+1>\left|A \cup B\right|$, there exists $x \in A \cup aU=B \cup bV$ such that $x \not\in A\cup B$.
This implies that $x \in aU \cap bV$.
We can also assume that $x$ is a leaf of $A \cup aU=B \cup bV$.
Then Lemma \ref{lem:crack} implies that there exist elements $\u',\V',\z \in S$ such that $\left| A \cup aU'\right|<\left|A \cup aU\right|$ and
\[
(\u',\V') \in \R(\a,\b),\, (\u,\V)=(\u',\V')\z.
\]
In this case the induction hypothesis implies that $(\u',\V') \in X\cdot S$, so that $(\u,\V) \in X \cdot  S$ as required.

For $(\r)$, the proof is entirely similar.
We show that the set
\[
Y=\{\u \in S: \a\u=\b\u,\, A\cup aU=A \cup B\}
\]
generates $\r(s,t)$, making particular use of the final statement of Lemma \ref{lem:crack}. 
\end{proof}

The free inverse monoid and the free ample monoid are left-right dual, so from  the dual of Lemma \ref{lem:crack} they satisfy $(\bf{L})$ and $(\bf{l})$. To show that
$\FLA$ satisfies $(\bf{L})$ and $(\bf{l})$, we first prove a result corresponding to Lemma~\ref{lem:crack}. 

\begin{Lem}\label{lem:crackleft}
Let $\u\a=\V\b$ in $\FLA$ and suppose that there exists $x \in U \cup uA=V \cup vB$ such that $x$ is either a leaf, or $x=\epsilon$ and every element of $(U \cup uA) \setminus \{\epsilon\}$ has a common nonempty prefix (this corresponds to a tree having a root with degree $1$).
Furthermore, suppose that $x \not\in uA \cup vB$.
Then there exist $\u',\V',\z \in \FLA$ such that $\left| U' \cup u'A\right|<\left|U \cup uA\right|$,
\[
\u'\a=\V'\b \text{ and } (\u,\V)=\z(\u',\V').
\]
Furthermore, if $\u=\V$ then $\u'=\V'$.
\end{Lem}

\begin{proof}  Note that as $x\notin  uA \cup vB$, $x\neq u$ and $x\neq v$.
If $x$ is a leaf, then let $\z=(x\da,1), U'=U \setminus \{x\}, u'=u, V'=V \setminus \{x\}, v'=v$.
In this case
\[
\u'\a=\big((U \cup uA) \setminus \{x\},ua\big)=\big((V \cup vB)\setminus \{x\},vb\big)=\V'\b, \z\u'=\u, \z\V'=\V.
\]
Furthermore, if $\u=\V$ then of course $\u'=\V'$.

If $x=\epsilon$ then $x \not \in uA \cup vB$ implies $u,v \neq \epsilon$.
Let $z$ be the common first letter of elements of $(U\cup uA) \setminus \{\epsilon\}$ and let $\z=(\{\epsilon,z\},z)$.
Then if we set $(U',u')=(z^{-1} (U \setminus \{\epsilon\}),z^{-1}u)$ and $(V',v')=(z^{-1} (V \setminus \{\epsilon\},z^{-1}v)$ then
\[
U' \cup u'A=z^{-1} ( U \setminus \{\epsilon\}) \cup z^{-1}uA=z^{-1}\big((U \cup uA) \setminus \{\epsilon\}\big)=\ldots=V' \cup v'B,
\]
which shows that $\u'\a=\V'\b$.
Also we have
\[
Z \cup zU'=\{\epsilon,z\} \cup (U\setminus \{\epsilon\})=U,
\]
because $z \in U$ (being the first letter of $u$).
As a consequence $\z\u'=\u$ and similarly $\z\V'=\V$ also.
Lastly, if $\u=\V$ then clearly $\u'=\V'$ which finishes the proof.
\end{proof}

\begin{prop}
The free inverse monoid $\FI$, the free ample monoid $\FA$ and the free left ample monoid $\FLA$  satisfy $(\bf{L})$ and $(\bf{l})$.
\end{prop}

\begin{proof} We have already mentioned that $\FI$ and  $\FA$ must satisfy $(\bf{L})$ and $(\bf{l})$.
For $\FLA$, let $\a,\b \in \FLA$.
Then either $\mathbf{L}(\a,\b)$ is empty or one of $a$ and $b$ is a suffix of the other.
Without loss of generality we can assume that $b=ya$ for some $y \in \Omega^*$.
In this case we claim that the finite set
\[
X=\{(\u,\V): \u\a=\V\b, U \cup uA = B \cup yA\}
\]
generates $\bf{L}(\a,\b)$.
Note that if $(\u,\V) \in \bf{L}(\a,\b)$ then necessarily $u=vy$ so from the equation $U\cup vyA=V \cup vB$ we conclude that $v (B \cup yA) \subseteq U\cup uA$.
As a consequence we see that if $\left| U \cup uA\right|\leq \left| B \cup yA\right|$ then $U \cup uA=v(B \cup yA)$, which implies that $v=\epsilon$ so that $U \cup uA=B \cup yA$ and $(\u,\V) \in X$.

Suppose now that there exists an $n \geq \left|B \cup yA\right|$ such that whenever $\left|U\cup uA\right|\leq n$ and $(\u,\V) \in \bf{L}(\a,\b)$, then necessarily $(\u,\V) \in \FLA \cdot X$.
Now let $(\u,\V) \in \bf{L}(\a,\b)$ be such that $\left|U \cup uA\right|=n+1$.
Note that $ua=vya$ implies that $u=vy$.
Then $U\cup vyA=V \cup vB$, so $v(B \cup yA) \subseteq U \cup vyA$.
However, $\left|v(B \cup yA)\right|=\left|B \cup yA\right| < \left| U \cup vyA\right|$, so $U \cup uA \neq v(B\cup yA)=uA \cup vB$.

If there exists a leaf of $U \cup uA$ which is not contained in $uA \cup vB$ then let $x$ be one such leaf.
However, if there is no such leaf then that means that every leaf of $U \cup uA$ is contained in $v(B \cup yA)$.
If $v=\epsilon$ then as $y\in B$, $v(B \cup yA)$ is  prefix closed so $U \cup uA=v(B \cup yA)=uA \cup vB$, which is a contradiction.
So $v \neq \epsilon$, and we have that all leaves of $U \cup uA$ have $v$ as a prefix.
This can only happen if $U \cup uA=v\da \cup vC$ for some prefix closed set $C$, which shows that every element of $(U \cup uA) \setminus \{\epsilon\}$ has the same first letter as $v$.
In this case let $x=\epsilon$.
Then Lemma \ref{lem:crackleft} implies that there exists $\u',\V',\z \in \FLA$ such that $\left|U' \cup u'A\right|<\left|U \cup uA\right|$,
\[
(\u',\V') \in \bf{L}(\a,\b) \text{ and } (\u,\V)=\z(\u',\V').
\]
In this case the induction hypothesis implies that $(\u',\V') \in \FLA \cdot X$ and so we have $(\u,\V) \in \FLA \cdot X$ as required.

For $(\bf{l})$, the proof is entirely similar, namely the finite set
\[
Y=\{U \in S: \u\a=\u\b, U \cup uA=B \cup yA\}
\]
generates $\bf{l}(\a,\b)$ if $b=ya$.

\end{proof}

\section{$\FLA$: analysis of $H$-sequences}\label{sec:positive}

In order to show that  $\FLA$ is right coherent, we make a careful examination of 
$H$-sequences for finite sets $H\subseteq \FLA \times \FLA$. 

\begin{Def} Let $\a\in \FLA$. 
\begin{enumerate}
\item[(i)] 
The \emph{weight} $w(\a)$ of ${\bf a}$ is defined by $w(\a)=\left|A\right|-1 + l(a)$. 
\item[(ii)] The \emph{diameter} $d(\a)$ of ${\bf a} $ is 
defined by $d(\a)=\text{max }\{l(u):u\in A\}$.
\end{enumerate}
\end{Def}

The following lemma states the most important basic properties of the weight function. 

\begin{Lem}\label{Basic}
Let  $\a,\b,\c,\a_1,\ldots,\a_n \in \FLA$. Then 
\begin{enumerate}
\item[\rm{(W0)}] $w(\a)=0$ if and only if $\a=\bf{1}$;
\item[\rm{(W1)}] $w(\a),w(\b)\leq w(\a\b)\leq w(\a)+w(\b)$;
\item[\rm{(W2)}] $w(\a\b)=w(\a)$ if and only if $\a \b= \a$, and this is equivalent to $\b\in E(\FLA) $ with $\a\leq_{\mathcal{L}} \b$.
\end{enumerate}
\end{Lem}
\begin{proof} The proof of (W0) is clear. 

For (W1), let
$\a=(A,a)$ and $\b=(B,b)$, so that $\a \b =(A\cup aB,ab)$. Then 
\[w(\a \b)=|A\cup aB|-1+l(ab)\]
and as $|A\cup aB|\geq |A|, |aB|$ where $|aB|=|B|$ and $l(ab)\geq l(a),l(b)$, we have
$w(\a),w(\b)\leq w(\a\b)$.

On the other hand, the second inequality for (W1) follows from the observation that
as $a\in A\cap aB$ we have \[|A\cup aB|=|A|+|aB\setminus A|\leq |A|+|aB|-1
=|A|+|B|-1.\] 

 Clearly $|A\cup aB|\geq |A|$ and
$l(ab)\geq l(a)$, so that if $w(\a\b)=w(\a)$, we must have
$|A\cup aB|=|A|$ and $l(b)=0$. Hence $b=\epsilon$, $aB\subseteq A$ and 
so $\a\b=\a$. 

If $\a\b=\a$ (equivalently, $w(\a\b)=w(\a)$), then we have shown that $\b\in E(${\em FLA}$(\Omega)) $ and clearly $\a\leq_{\mathcal{L}} \b$. The converse is clear. Thus (W2) holds.
\end{proof}

The proof of our main result depends heavily on the fact that certain factorisations can be carried through sequences.
The following two lemmas constitute the foundations of this process.

\begin{Lem} \label{Crack}
Let $\d\z=\b\V$, $\z \neq {\bf 1}$ and let $x$ be a leaf of $Z$ such that $dx \not\in B$.
Then there exist elements $\z',\x,\V' \in \FLA$ such that 
\[
Z'=Z \setminus \{x\}, w(\z')<w(\z),\,\z=\z'\x,\,\V=\V'\x,\, \d\z'=\b\V'
\]
and
\begin{enumerate}

\item if $x\neq z$ and $dx \not\in D$ then $\x=(\tilde{x}\da \cup \tilde{z}\da,\tilde{z}), \V'=(V \setminus \{b^{-1}dx\},v\tilde{z}^{-1})$ where $\tilde{x},\tilde{z} \in \Omega^*$ have no common non-empty prefix, $x=z'\tilde{x}, z=z'\tilde{z}$ (so $dx=dz'\tilde{x}=bv'\tilde{x}$),

\item if $x=z$ (then necessarily $x\neq \epsilon$)  and $dx \not\in D$ then $\z'=(Z',zx'^{-1}), \x=(\{\epsilon,x'\},x')$ and $\V'=(V\setminus \{v\},vx'^{-1})$, where $x'$ is the last letter of $x$,

\item if $x=z$ (then necessarily $x\neq \epsilon$) and $dx \in D$ then $\z'=(Z',zx'^{-1}), \x=(\{\epsilon,x'\},x')$ and $\V'=(V,vx'^{-1})$, where $x'$ is the last letter of $x$,

\item if $x\neq z$ and $dx \in D$ then $\z'=(Z',z'), \x=(\tilde{x}\da \cup \tilde{z}\da,\tilde{z}), \V'=(V,v\tilde{z}^{-1})$ where $\tilde{x},\tilde{z} \in \Omega^*$ have no common non-empty prefix.

\end{enumerate}

Furthermore, the following are true:
\begin{enumerate}
\item[(A)] in cases $(1)$ and $(2)$ we have $\left|D \cup dZ'\right|<\left|D \cup dZ\right|$ and that if $\z=\V$ then $\z'=\V'$,
\item[(B)] in cases $(1),(2)$ and $(3)$ we have $w(\b\V')=w(\d\z')<w(\d\z)=w(\b\V)$.
\end{enumerate}
\end{Lem}

\begin{proof}
We investigate all $4$ cases separately:

Case (i): $dx \not\in D$ and $x \neq z$.
Let $z'$ be the greatest common prefix of $z$ and $x$, that is, there exist $\tilde{z}$ and $\tilde{x}$ such that $z=z'\tilde{z}$ and $x=z'\tilde{x}$ and $\tilde{z}$ and $\tilde{x}$ have no common non-empty prefix.
It is important to note that $\tilde{x}\neq \epsilon$, for $x$ is a leaf different from $z$.
Now let
\[
\z'=(Z \setminus \{x\},z'), \x=(\tilde{x}\da \cup \tilde{z}\da,\tilde{z}).
\]
Then it is easy to check that $\z',\x \in \fla$ and $\z=\z'\x$.
Note that since $dx \not\in B$, but $dx \in B \cup bV$, we have that $dx =dz'\tilde{x} \in bV$, and that $bv=dz=dz'\tilde{z} \in bV$ also.
Since $\tilde{z}$ and $\tilde{x}$ have no common non-empty prefix, we conclude that $b$ is a prefix of $dz'$.
As a consequence of the fact that $bv=dz'\tilde{z}$, we conclude that $\tilde{z}$ is a suffix of $v$, so $v\tilde{z}^{-1} \in V$.
Furthermore, $bv=dz'\tilde{z}$ implies that $v\tilde{z}^{-1}=b^{-1}dz'\neq b^{-1}dz'\tilde{x}=b^{-1}dx$.
Now let
\[
\V'=(V \setminus \{b^{-1}dx\},v\tilde{z}^{-1}).
\]
Note that our assumption that $dx \not\in D$ implies that $dx$ is a leaf of $B\cup bV$.
Then, since $dx \not \in B$, we have that $b^{-1}dx$ is a leaf of $V$, so $\V' \in \fla$.
It is then easy to check that $\V=\V' \x$, since the second coordinates are the same, and $b^{-1}dx=b^{-1}dz'\tilde{x}=v\tilde{z}^{-1} \tilde{x}$.
Similarly $\d\z'=\b\V'$, for the second coordinates are both equal $dz'$, and the first coordinates both equal $(B \cup bV) \setminus \{dx\}$.
Also we have that $w(\b\V') < w(\b\V)$, because $dx \in B \cup bV$.
Furthermore, if $\z=\V$ then from $\d\z=\b\V$ we conclude that $d=b$ which implies that $b^{-1}dx=x$.
Similarly $v\tilde{z}^{-1}=b^{-1}dz'=z'$, showing that $\z'=\V'$.

Case (ii): $dx \not\in D$, and $x=z$.
We have that $z \neq \epsilon$, for otherwise $\z={\bf 1}$.
So let $z=z'x'$ where $x' \in \Omega$, and let
\[
\z'= (Z \setminus \{z\},z'),\ \x=(\{\epsilon,x'\},x').
\]
We have that $\z',\x \in \fla$, and that $\z=\z'\x$.
Note that $dz \not \in B$, but it is the second coordinate of $\b\V$.
Thus, $v \neq \epsilon$, and we have that $x'$ is the last letter of $v$ and as a consequence, $dz'=bv'$, where $v'=v(x')^{-1}$.
We see that $v$ is a leaf of $V$ and similarly to the previous case it is easy to show that if we define
\[
\V'=(V \setminus \{v\},v'),
\]
then $\V' \in \fla, w(\b\V')<w(\b\V), \V=\V'\x$ and $\d\z'=\b\V'=\big( (D \cup dZ) \setminus\{dz\},dz'\big)$.
Furthermore, if $\z=\V$ then of course $z=v$ and we conclude that $\z'=\V'$, so the statements of the lemma are true.

Case (iii): $dx \in D$, and $x=z$.
This case is similar to Case (ii), the only difference being that we have to define
\[
\V'=(V,v').
\]
Since the second coordinate of $\b\V'$ is one letter shorter than $bv$,
we have that $w(\b\V')<w(\b\V)$.

Case (iv): $dx \in D$ and $x \neq z$.
Put
\[
\z'=(Z \setminus \{x\},z'),\ \x=(\tilde{x}\da \cup \tilde{z}\da,\tilde{z}) \text{ and }  \V'=(V,v\tilde{z}^{-1})
\]
where $z',\tilde{z}$ and $\tilde{x}$ are defined as in Case (i).
It is easy to check (using the same argument as in Case (i)) that $b^{-1}dx=v\tilde{z}^{-1}\tilde{x}$ is a leaf in $V$, $\z',\x,\V' \in \fla$, $w(\z')<w(\z)$ and
\[
\z=\z'\x,\ \V=\V'\x \text{ and }\d\z'=\b \V',
\]
so that again, the statements of the lemma are true.

\end{proof}

\begin{Lem} \label{Roll}
Let $\a\b=\c\d$ such that $\b=(x\da \cup b\da,b)$ for some $b,x \in \Omega^*,x\neq\epsilon$, having no common non-empty prefix.
If $ax \not\in A \cup C$ and $A=(A\cup aB) \setminus \{ax\}$, then $\d=\d'\b$ for some $\d'=(D\setminus\{ d'x\},d')$ such that  $\a=\c\d'$.
\end{Lem}

\begin{proof} First remark that our hypotheses guarantee that $ax$ is a leaf of $A\cup aB=C\cup cD$.

Since $ab=cd$, $c$ is a prefix of $ab$.
However, since $ax \in C \cup cD$, but $ax \not\in C$, we have that $c$ is also a prefix of $ax$.
Since $b$ and $x$ have no common non-empty prefix, this implies that $c$ is a prefix of $a$.

Let   $d' \in \Omega^*$ be such that $a=cd'$.
We have that $ax=cd'x \in cD$, so $d'x \in D$.
From $cd'b=ab=cd$ we deduce that $d'b=d\in D$.
From $d'b,d'x \in D$, the prefix closure of $D$ gives  that $d'B \subseteq D$.
Observe now that $d'x$ is a leaf of $D$ and $d'x\neq d'$, so that $\d'=(D\setminus\{ d'x\},d')\in \fla$ and clearly, $cd'x \not\in C \cup cD'$.
Moreover, it is easy to check that 
\[\a=\c\d'\mbox{ and } \d=\d'\b.\]
\end{proof}

Let $\rho$ be a finitely generated right congruence on $\fla$.
Without loss of generality we may suppose that  $\rho=\langle H \rangle$ for some finite $H \subseteq \fla\times \fla$ with $H^{-1}= H$.
Let us denote by $\mathcal{D}$ the maximum of the diameters of the components of the elements of $H$. In the following definition, we abuse terminology a little. The elements $\a,\u,\b$ and $\mathbf{v}$ play a special role, but are not distinguished from the products $\a \u$ and $\b \mathbf{v}$. We employ similar conventions in other circumstances.

\begin{Def}
Suppose that we have an $H$-sequence
\[
\a\u=\c_1\t_1,\d_1\t_1=\c_2\t_2,\ldots,\d_n\t_n=\b\V
\]
connecting $\a\u$ and $\b\V$.
Then we say that the $H$-sequence is \emph{reducible} if there exist elements $\y,\u',\t_1',\ldots,\t_n',\V'$ such that
\begin{itemize}
\item[(Red1)] $w(\a\u')<w(\a\u)$, $w(\b\V')<w(\b\V)$ or $w(\t_i')<w(\t_i)$ for some $i$;
\item[(Red2)] $\u=\u'\y, \t_1=\t_1'\y,\ldots,\t_n=\t_n'\y,\V=\V'\y$;

\item[(Red3)] $\a\u'=\c_1\t_1',\d_1\t_1'=\c_2\t_2',\ldots,\d_n\t_n'=\b\V'$.
\end{itemize}

If a sequence is not reducible, we call it \emph{irreducible}.
\end{Def}

From the above definition, a length-0
 $H$-sequence $\a\u=\b\V$ is reducible if and only if there exist elements $\y,\u',\V' \in \fla$ such that $\u=\u'\y, \V=\V'\y, \a\u'=\b\V'$ and $w(\a\u')=w(\b\V')<w(\a\u)=w(\b\V)$.

Note  that if (Red2) holds, then in view of (W2) in Lemma~\ref{Basic},  (Red1) is equivalent to saying that $\a\u'\neq \a\u$,  $\b\V' \neq \b\V$ or $\t_i'\neq \t_i$ for some $i$ - we are going to make use of this fact in the sequel.
We are going to show that every irreducible sequence has an element with diameter less than or equal to $2\m(\mathcal{D},d(\a),d(\b))$.

\begin{Lem}\label{Two}
If the sequence $\a\u=\b\V$ is irreducible then $d(\u)\leq \m(d(\a),d(\b))$.
\end{Lem}

\begin{proof}
Suppose that $d(\u)>d(\a),d(\b)$.
Then there exists a leaf $x \in U$ such that $l(u)>d(\a),d(\b)$.
As a consequence we have $ax \not\in A \cup B$, so by Cases $(1)$ and $(2)$ of Lemma \ref{Crack} there exist $\u',\V',\x \in \fla$ such that $\a\u'=\b\V', \u=\u'\x,\V=\V'\x$ and $w(\b\V')<w(\b\V)$, contradicting the irreducibility of the sequence $\a\u=\b\V$.
\end{proof}

The following Lemma shows that elements of $\FLA$ which are connected by an irreducible sequence are `lean' - the length of their second component limits their diameter.
In fact, much more is true, but this statement will suffice for our proof.
Furthermore, it is worth noting that this lemma is one (the other one is Statement (\ref{P1}) of Lemma \ref{Crack}) which is not dualisable - it fails if we swap from right congruences to left congruences.

\begin{Lem} \label{Small}
If
\begin{equation}\label{Seqq}
\a\u=\c_1\t_1,\d_1\t_1=\c_2\t_2,\ldots,\d_n\t_n=\b \V
\end{equation}
is an irreducible sequence, then $d(\a\u) \leq 2\m(l(au),d(\a),d(\b),\mathcal{D})$.
\end{Lem}

\begin{proof}
Let $\M=\m(l(au),d(\a),d(\b),\mathcal{D})$.
For brevity let $\c_{n+1}=\b$ and $\t_{n+1}=\V$.
Suppose that $d(\a\u)>2\M$, which clearly implies that $\u\neq {\bf 1}$.
Let $y$ be a leaf of $A \cup aU$ with $l(y)=d(\a\u)>2\M$.
Then clearly $y \not\in A$, so $y=ax$ for some leaf $x \in U$.
Notice that  since $l(a)\leq d(\a)$, we have that $l(x)>\M \geq d(\a),d(\c_1)$, so $ax \not\in A \cup C_1$.
Also, $l(ax)>l(au)$ implies that  $x \neq u$.
Then if we apply Lemma \ref{Crack} to the equality $\a\u=\c_1\t_1$ and the leaf $x \in U$, we obtain by Case $(1)$ that there exist  elements $\x,\u',\t_1' \in \fla$ such that
\[
w(\a\u')<w(\a\u),\,\u=\u'\x,\,\t_1=\t_1'\x,\a\u'=\c_1\t_1',
\]
\[
\x=(\tilde{x}\da \cup \tilde{u}\da,\tilde{u}) \text{ and } \t_1'=(T_1\setminus\{ t_1'\tilde{x}\},t_1')\]
 with $\tilde{x},\tilde{u} \in \Omega^*$ having no common non-empty prefix and $x=u'\tilde{x}$.
Note that $ax=au'\tilde{x}$, $l(ax)>2\M\geq \M+l(au)$ and $au'$ is a prefix of $au$, so we have that $l(\tilde{x})>\M$. Further,
$C_1 \cup c_1T_1'=(C_1 \cup c_1T_1) \setminus \{c_1t_1'\tilde{x}\}$. 

Note that if $n=0$ then we have already  contradicted the irreducibility of the sequence (\ref{Seqq}), so in the sequel we suppose that $n>0$.

Suppose for induction that we have constructed elements $\u',\t_1',\ldots,\t_m'\in \fla$ satisfying $\u=\u'\x$, $\t_i=\t_i'\x$ for all $1\leq i\leq m$, $T_m'=T_m \setminus \{t_m'\tilde{x}\}$ and $C_m \cup c_mT_m'=(C_m \cup c_mT_m) \setminus \{c_mt_m'\tilde{x}\}$.

Since $l(\tilde{x})>\M$, we have that $d_mt_m'\tilde{x} \not \in (D_m \cup d_mT_m') \cup C_{m+1}$, so $D_m \cup d_mT_m'=(D_m \cup d_mT_m) \setminus \{d_mt_m'\tilde{x}\}$.
We can therefore apply Lemma \ref{Roll} to the equality $\d_m\t_m' \cdot \x=\c_{m+1}\t_{m+1}$ and obtain that $\t_{m+1}=\t_{m+1}'\x$ for some $\t_{m+1}'$
with $T_{m+1}'=T_{m+1} \setminus \{t_{m+1}\tilde{x}\}$ and $\d_m\t_m'=\c_{m+1}\t_{m+1}'$, so that $C_{m+1} \cup c_{m+1}T_{m+1}'=(C_{m+1} \cup c_{m+1}T_{m+1}) \setminus \{c_{m+1}t_{m+1}\tilde{x}\}$.

Applying induction (note that $\M\geq d(\b)$ is required at the last step), there exist elements $\u',\t_1',\ldots,\t_n',\V'$ such that $\u=\u'\x,\t_1=\t_1'\x,\ldots,\t_n=\t_n'\x,\V=\V'\x, w(\a\u')<w(\a\u)$ and

\[\a\u'=\c_1\t_1',\d_1\t_1'=\c_2\t_2',\ldots,\d_n\t_n'=\b \V'.\]
This contradicts the irreducibility of the sequence (\ref{Seqq}) and so we conclude that $d(\a\u) \leq 2\M$. 
\end{proof}

\begin{Def}
We say that the pair  $(\a \u,\b \V)$ is \emph{irreducible} if $\a \u$ and $\b \V$ can be connected by an irreducible $H$-sequence.
\end{Def}

Note that in view of an earlier remark, we are a little cavalier above; more properly, we should write  $\a \cdot \u$ and $\b\cdot \V$.

\begin{Def} 
Let $\a\u=\c_1\t_1, \d_1\t_1=\c_2\t_2,\ldots,\d_n\t_n=\b\V$ be an $H$-sequence $\mathcal{S}$. We define the
{\em weight} $w$ of $\mathcal{S}$ to be 
$w(\a\u)+w(\t_1)+\ldots+w(\t_n)+w(\b\V)$.
\end{Def} 

\begin{Lem} \label{Irr}
Let 
\[\mathcal{S}: \a\u=\c_1\t_1, \d_1\t_1=\c_2\t_2,\ldots,\d_n\t_n=\b\V\] be an $H$-sequence.
Then there exist elements $\y,\u',\t_1',\ldots,\t_n',\V'$ such that
\[
\u=\u'\y, \t_1=\t_1'\y,\ldots,\t_n=\t_n'\y, \V=\V'\y,
\]
and
\[
\a\u'=\c_1\t_1',\d_1\t_1'=\c_2\t_2',\ldots,\d_n\t_n'=\b\V'
\]
is an irreducible $H$-sequence.
\end{Lem}
\begin{proof} We use induction on the weight of 
$\mathcal{S}$. First note that by Lemma~\ref{Basic}, $w(\mathcal{S})\geq w(\a)+w(\b)$.

If $w(\mathcal{S})=w(\a)+w(\b)$, then again by Lemma~\ref{Basic} we have that
$\a\u=\a$, $\b\V=\b$ and $w(\t_1)=\ldots=w(\t_n)=0$, so that $\t_1=\ldots=\t_n={\bf 1}$ and our $H$-sequence is irreducible in view of (Red1).

Suppose now that $w(\mathcal{S})>w(\a)+w(\b)$ and the $H$-sequence
\[
\a\u=\c_1\t_1, \d_1\t_1=\c_2\t_2,\ldots,\d_n\t_n=\b\V
\]
is reducible.
Then there exist elements $\tilde{\y},\tilde{\u},\tilde{\t}_1,\ldots,\tilde{\t}_n,\tilde{\V}$ satisfying conditions (Red1)-(Red3), that is, $\u=\tilde{\u}\tilde{\y},\t_i=\tilde{\t}_i\tilde{\y}$ for all $1 \leq i \leq n$, $\V=\tilde{\V}\tilde{\y}$,
\begin{equation}\label{Nyaff}
\a\tilde{\u}=\c_1\tilde{\t}_1,\d_1\tilde{\t}_1=\c_2\tilde{\t}_2,\ldots,\d_n\tilde{\t}_n=\b\tilde{\V}
\end{equation}
and
\[
w(\a\tilde{\u})+w(\tilde{\t}_1)+\ldots+w(\tilde{\t}_n)+w(\b\tilde{\V})<w(\a\u)+w(\t_1)+\ldots+w(\t_n)+w(\b\V).
\]
This inequality shows that we can apply the inductive hypothesis to the $H$-sequence (\ref{Nyaff}).
Thus there exists an irreducible sequence
\[
\a\u'=\c_1\t_1',\ldots,\d_n\t_n'=\b\V'
\]
and an element $\y'$ such that $\tilde{\u}=\u'\y',\tilde{\t}_i=\t_i'\y'$ and $\tilde{\V}=\V'\y'$.
In this case let $\y=\y'\tilde{\y}$, and the lemma is proved.

\end{proof}

This lemma shows that if $(\a\u,\b\V)$ is not irreducible, then it is a `direct consequence' of an irreducible pair $(\a \u',\b\V')$.
The following lemma will be used to `dismantle' irreducible sequences, and to show that they always contain a `small' element.

\begin{Lem}\label{New}
Let 
\begin{equation}\label{BasicS}
\a\u=\c_1\t_1,\ldots,\d_{n-1}\t_{n-1}=\c_n\t_n, \d_n\t_n=\b\V
\end{equation}
be an irreducible sequence.
Then there exist $\z, \u',\t_1',\ldots, \t_n' \in \fla$ such that
\begin{equation}\label{P1}
d(\z)\leq \m(d(\a),d(\b),\mathcal{D}),
\end{equation}
\begin{equation}\label{P2}
\u=\u'\z, \t_1=\t_1'\z, \ldots, \t_n=\t_n'\z,
\end{equation}
and such that the sequence
\begin{equation}\label{P3}
\a\u'=\c_1\t_1', \ldots, \d_{n-1}\t_{n-1}'=\c_n\t_n'
\end{equation}
is irreducible.
Furthermore, if $\z \neq {\bf 1}$, then
\begin{equation}\label{Min}
\text{min}(d(\a\u),d(\b\V)) \leq 2\m(d(\a),d(\b),\mathcal{D}).
\end{equation}
\end{Lem}

\begin{proof}
If the sequence
\begin{equation}\label{Chopped}
\a\u=\c_1\t_1, \ldots, \d_{n-1}\t_{n-1}=\c_n\t_n
\end{equation}
is irreducible then $\z={\bf 1}, \u=\u', \t_i'=\t_i$ for $1\leq i\leq n$ satisfy the requirements of the lemma.
Let us therefore suppose that the sequence (\ref{Chopped}) is reducible.
Then by Lemma \ref{Irr} there exist $\z\neq {\bf 1},\u',\t_1',\ldots, \t_n' \in \fla$ such that (\ref{P2}) and (\ref{P3}) are satisfied.

Let us fix $\u',\t_1',\ldots,\t_n'$, and choose a $\z$ such that its weight is minimal amongst those satisfying the equalities (\ref{P2}).
We claim that this particular $\z$ satisfies (\ref{P1}) by first showing that $Z \subseteq (au')^{-1}A \cup (d_nt_n')^{-1}B$ where
\[
g^{-1} X=\{y \in \Omega^*: gy \in X\}.
\]
Note that if $X$ is prefix closed then so is $g^{-1}X$.
Therefore it is enough to show that the leaves  of $Z$ are contained in $(au')^{-1}A \cup (d_nt_n')^{-1}B$.
Let $x$ be a leaf of $Z$, and suppose that $d_nt_n'x \not\in B$.

Then by applying Lemma \ref{Crack} to the equation $\d_n\t_n' \cdot \z=\b\cdot \V$, there exist elements $\z',\V',\x \in \fla$ such that $\z=\z'\x, w(\z')<w(\z), \V=\V'\x$ and $\d_n\t_n'\z'=\b\V'$.
If we multiply the sequence (\ref{P3}) by $\z'$ and combine it with the equality $\d_n\t_n'\z'=\b\V'$ we obtain the $H$-sequence
\begin{equation}\label{Multi}
\a\u'\z'=\c_1\t_1'\z', \ldots, \d_{n-1}'\t_{n-1}'\z'=\c_n\t_n'\z', \d_n\t_n'\z'=\b\V'.
\end{equation}
Note that if we multiply the sequence (\ref{Multi}) by the element $\x$ we obtain the sequence (\ref{BasicS}). 

If $x=z$ or $d_nt_n'x \not \in D_n \cup d_nT_n'$, then we also have that $w(\b\V')<w(\b\V)$, contradicting the irreducibility of sequence (\ref{BasicS}). 

We therefore conclude that $x\neq z$ and $d_nt_n'x \in D_n \cup d_nT_n'$. 
Since sequence (\ref{BasicS}) is irreducible, this can only happen if $\a\u'\z'=\a\u, \t_1'\z'=\t_1,\ldots \t_n'\z'=\t_n$ and $\b\V'=\b\V$.
Note that $w(\z')<w(\z)$, so by the minimality of $w(\z)$, one of the equations of (\ref{P2}) must fail for $\z'$, and since we have just shown that $\t_i=\t_i'\z'$ for all $i$, we have that $\u\neq \u'\z'$.
Notice that $\a\u'\z'=\a\u$ implies that the second coordinates of $\u$ and $\u'\z'$ are the same and so the first coordinates of $\u$ and $\u'\z'$ are different.
Since $\z'=(Z \setminus \{x\},z')$, the first coordinate of $\u'\z'$ can differ from the first coordinate of $\u=\u'\z$ only in the element $u'x$.
That is, $u'x \not\in U' \cup u'Z'$.
However, $\a\u=\a\u'\z'$ and $au'x \in A \cup aU$, so $au'x \in A \cup a(U' \cup u'Z')$, that is, $au'x \in A$.

So far we have shown that for every leaf $x$ of $Z$, if $d_nt_n'x \not \in B$, then $au'x \in A$.
This shows that every leaf $x$ of $Z$ is contained in the prefix closed set $(au')^{-1}A \cup (d_nt_n')^{-1}B$, so $Z \subseteq (au')^{-1}A \cup (d_nt_n')^{-1}B$.
Since $d(g^{-1}X)\leq d(X)$ for every $g \in \Omega^*$ and finite $X \subseteq \Omega^*$, we conclude that $d(\z)\leq \text{max}(d(\a),d(\b))\leq \m(d(\a),d(\b),\mathcal{D})$.

We have observed that $\z\neq \mathbf{1}$.
Either $au'z \in A$ or $d_nt_n'z \in B$.
If $d_nt_n'z \in B$ then $l(bv)=l(d_nt_n)=l(d_nt_n'z)\leq d(\b)$, whilst if $au'z \in A$, then $l(au)=l(au'z) \leq d(\a)$.
Lemma \ref{Small} implies in the first case that $d(\b\V) \leq 2\m(d(\a),d(\b),\mathcal{D})$, whilst in the second case $d(\a\u) \leq 2\m(d(\a),d(\b),\mathcal{D})$.

\end{proof}

As a consequence of this lemma we can show that every irreducible sequence contains a `small' element.

\begin{Lem}\label{Sm}
Let
\begin{equation}\label{SeqSm}
\a\u=\c_1\t_1,\ldots,\d_n\t_n=\b\V
\end{equation}
be an irreducible $H$-sequence.
Then there exists an element in the sequence having diameter less than or equal to $2\m(d(\a),d(\b),\mathcal{D})$.
\end{Lem}

\begin{proof}
Let $\mathcal{D}'=\m(d(\a),d(\b),\mathcal{D})$.
If $d(\a\u) \leq 2 \mathcal{D}'$, then the statement is true, so let us suppose that $d(\a\u) >  2\mathcal{D}'$.

Apply Lemma \ref{New} to the sequence (\ref{SeqSm}).
Note that $\z\neq 1$ if and only if the shortened sequence
\[
\a\u=\c_1\t_1,\ldots,\d_{m-1}\t_{m-1}=\c_m\t_m
\]
is also irreducible.
In this case we can apply Lemma \ref{New} to this shortened sequence, and repeat the procedure until $\z \neq \bf{1}$.
Note that such a $\z$ exists, for otherwise we would have that the sequence $\a\u=\c_1\t_1$ is irreducible, which by Lemma \ref{Two} contradicts our assumption that $d(\a\u)>2\mathcal{D}'$.
That is, there exists $2 \leq i\leq n+1$ such that
\[
\a\u=\c_1\t_1,\ldots, \d_{j-1}\t_{j-1}=\c_j\t_j
\]
is irreducible for all $i\leq j\leq n+1$ (where we denote $\b$ by $\c_{n+1}$ and $\V$ by $\t_{n+1}$), but
\[
\a\u=\c_1\t_1,\ldots, \d_{i-2}\t_{i-2}=\c_{i-1}\t_{i-1}
\]
is reducible.
In this case if we apply Lemma \ref{New} to the first sequence with $j=i$, then the acquired element $\z$ will be different from ${\bf 1}$, and as a consequence the lemma implies that $\text{min}(d(\a\u),d(\c_i\t_i)) \leq 2 \mathcal{D}'$.
\end{proof}

Now let
\begin{equation}\label{Seq1}
\a\u=\c_1\t_1,\ldots,\d_{n-1}\t_{n-1}=\c_n\t_n,\d_n\t_n=\b\V
\end{equation}
be an irreducible $H$-sequence with $n\geq 1$ and let $\mathcal{D}'=\m(d(\a),d(\b),\mathcal{D})$.
Then by Lemma \ref{New} there exist $\z,\u',\t_1',\ldots,\t_n'\in \fla$, $d(\z) \leq \mathcal{D}'$ such that $\u=\u'\z$ and $\t_i=\t_i'\z$ for every $1\leq i\leq n$, and such that the sequence
\[
\a\u'=\c_1\t_1',\ldots,\d_{n-1}\t_{n-1}'=\c_n\t_n'
\]
is irreducible.
Now let us apply Lemma \ref{New} to this sequence.
Thus, there exist elements $\y^{(n)},\u^{(n)},\t_1^{(n)},\ldots,\t_{n-1}^{(n)} \in \fla$, $d(\y^{(n)})\leq \mathcal{D}'$ satisfying $\u'=\u^{(n)}\y^{(n)}$, $\t_i'=\t_i^{(n)}\y^{(n)}$ for every $1\leq i\leq n-1$ and such that the $H$-sequence
\begin{equation}
\a\u^{(n)}=\c_1\t_1^{(n)},\ldots, \d_{n-2}\t_{n-2}^{(n)}=\c_{n-1}\t_{n-1}^{(n)}
\end{equation}
is irreducible.

Note that $\u=\u^{(n)}\y^{(n)}\z$ and $\t_i=\t_i^{(n)}\y^{(n)}\z$ for every $1\leq i\leq n-1$.
Inductively, for every $2\leq k \leq n$ we can define the elements $\u^{(k)},\y^{(k)}$ and $\t_i^{(k)}$ where $1 \leq i\leq k-1$ satisfying $\u^{(k+1)}=\u^{(k)}\y^{(k)}$ and $\t_i^{(k+1)}=\t_i^{(k)}\y^{(k)}$ for every $1\leq i\leq k-1$ such that the $H$-sequence
\begin{equation}\label{KSeq}
\a\u^{(k)}=\c_1\t_1^{(k)},\ldots, \d_{k-2}\t_{k-2}^{(k)}=\c_{k-1}\t_{k-1}^{(k)}
\end{equation}
is irreducible, and $d(\y^{(k)})\leq \mathcal{D}'$.

The last step is to define $\y^{(1)}$: at this point we have that the $H$-sequence
\begin{equation}
\a\u^{(2)}=\c_1\t_1^{(2)}
\end{equation}
is irreducible.
By Lemma \ref{Two}, we have that $d(\u^{(2)})\leq \m(d(\a),d(\c_1)) \leq \mathcal{D}'$.
So if we define $\y^{(1)}=\u^{(2)}$ then $d(\y^{(1)})\leq \mathcal{D}'$ .
For later reference, we summarise the properties of the elements $\y^{(i)}_j$ in the following lemma.

\begin{Lem}\label{lem:thedecomp}
If
\[
\a\u=\c_1\t_1,\ldots,\d_{n-1}\t_{n-1}=\c_n\t_n,\d_n\t_n=\b\V
\]
is an irreducible $H$-sequence with $n\geq 1$, then there exist elements $\z,\u^{(i)},\y^{(i)}$ and $\t^{(i)}_j$ where $1\leq j<i\leq n$ such that
\begin{itemize}
\item[\rm (Y1)] $\u=\y^{(1)} \ldots \y^{(n)}\z$, $\u^{(i)}=\y^{(1)}\ldots \y^{(i-1)}$ for every $2\leq i\leq n$,
\item[\rm (Y2)] $\t^{(j)}_i=\t^{(j-1)}_i \y^{(j-1)}$,
\item[\rm (Y3)] the $H$-sequence
\[
\a\u^{(j)}=\c_1\t^{(j)}_1,\ldots,\d_{j-2}\t^{(j)}_{j-2}=\c_{j-1}\t^{(j)}_{j-1}
\]
is irreducible for every $2\leq j\leq n$,
\item[(Y4)]  $d(\z), d(\y^{(i)})\leq \m(d(\a),d(\b),\mathcal{D})$ for all $1 \leq i \leq n$.
\end{itemize}
\end{Lem}

Notice  that for every $1\leq i\leq n$ we have that either $\a\y^{(1)}\ldots \y^{(i)}\neq \a\y^{(1)} \ldots \y^{(i+1)}$ or $\y^{(i+1)}$ is an idempotent (here we assume that $\y^{(n+1)}=\z$).

\section{The free left ample monoid and right coherency}\label{sec:flacoherent}

We are now in a position to show that $\FLA$ is right coherent. Assume first that $\Omega$ is finite. Continuing from 
Lemma~\ref{lem:thedecomp}, let $\mathcal{W}$ be the maximal weight of elements of $\fla$ having diameter less than or equal to $\mathcal{D}'$.
Since $\Omega$ is finite, so $\mathcal{W}$ exists.
If we multiply any number of idempotents having diameter less than or equal to $\mathcal{D}'$, then the diameter of the resulting element will be less than or equal to $\mathcal{D}'$, so the weight of the product will be less than or equal to $\mathcal{W}$.

Now let us `merge' the consecutive idempotents of the sequence $\y^{(1)},\ldots, \y^{(n)},\z$ with the succeeding non-idempotent elements. That is, if $\y^{(1)}$ is not idempotent, then let $\y_1=\y^{(1)}$. Otherwise, let $\y^{(1)}\ldots \y^{(i)}$ be the first maximal idempotent subsequence, and let $\y_1= \y^{(1)}\ldots \y^{(i)} \y^{(i+1)}$, and so on: if the next element is not idempotent, it will be $\y_2$, otherwise $\y_2$ will be the product of the following maximal subsequence of idempotents multiplied by the next non-idempotent. In case $\z$ is idempotent, the last element of the sequence $\y_1,\ldots,\y_m$ will be idempotent, but all the others are non-idempotent.
Notice that for every $1\leq i\leq m$, $\y_i$ is a product of idempotents followed by a non-idempotent except (possibly) in the case $i=m$.
All factors of $\y_i$ have diameter less than or equal to $\mathcal{D}'$, so the product of their diameters also has this property.
This implies that $w(\y_i)\leq \mathcal{W}$.
The properties of the sequence $\y_1,\ldots,\y_m$ are summarised in the following lemma.

\begin{Lem} \label{Why}
If
\[
\a\u=\c_1\t_1,\ldots,\d_n\t_n=\b\V
\]
is an irreducible $H$-sequence, then there exist elements $\y_1,\ldots,\y_m$ such that
\begin{itemize}
\item[\rm (C1)] $\u=\y_1\y_2\ldots \y_m$,
\item[\rm (C2)] $w(\y_i)\leq \mathcal{W}$ for every $1\leq i\leq m$, where $\mathcal{W}$ denotes the maximal weight of elements of $\fla$ having diameter less than or equal to $\m(d(\a),d(\b),\mathcal{D})$,
\item[\rm (C3)]  $\y_i$ is not an idempotent for all $1\leq i\leq m-1$,
\item[\rm (C4)] For every $1\leq i\leq m-1$, there exists an irreducible $H$-sequence connecting $\a\y_1\y_2\ldots \y_i$ with an element of the form $\c_i\tilde{\t}_i$ where $(\c_i,\d_i) \in H$.
\end{itemize}
\end{Lem}

We aim to show that the right annihilator congruence
\[
r(\a\rho)=\{(\u,\V) \in \fla \times \fla:\a\u \mathrel{\rho} \a\V\}
\]
is finitely generated for all $\a \in \fla$.
To show this, let $\a \in \fla$ be fixed.
Now let
\[
\mathbb{K}= \{\a\u\rho:\exists\, \b\V \in \fla\text{ with }d(\b)\leq \m(d(\a),\mathcal{D}) \text{ and } (\a\u,\b\V) \text{  irreducible}\}.
\]

\begin{Lem}\label{K}
The set $\mathbb{K}$ is finite.
\end{Lem}

\begin{proof}
Let $\a\u\rho \in \mathbb{K}$ and let
\[
\a\u=\c_1\t_1,\ldots,\d_n\t_n=\b\V
\]
be an irreducible $H$-sequence connecting $\a\u$ to an element $\b\V \in \fla$ testifying $\a\u\rho \in \mathbb{K}$.
Then by Lemma \ref{Sm} there exists  an element in the sequence having diameter less than or equal to $2\m(d(\a),\mathcal{D})$.
Since there are only finitely many such elements of $\fla$, we have that $\mathbb{K}$ is finite.
\end{proof}

Now let $\mathcal{K}=\left|\mathbb{K}\right|$, and let us define the set
\[
H'=\{(\u,\V):\a\u \mathrel{\rho} \a\V \text{ and }w(\a\u),w(\a\V) \leq (\mathcal{K}+3)\mathcal{W}'\},
\]
where $\mathcal{W}'$ is the maximum of the weights of elements of $\fla$ having diameter less than or equal to $2\m(d(\a),\mathcal{D})$.

\begin{Lem}\label{A}
The finite set $H'$ generates the right annihilator congruence of $\a\rho$.
\end{Lem}

\begin{proof}
Denote the right annihilator congruence of $\a\rho$ by $\tau$.
By definition, $H'\subseteq \tau$.
Now let $(\u,\V) \in \tau$.
We are going to show that $(\u,\V) \in \langle H'\rangle$.
Without loss of generality we can suppose that $w(\a\u)\geq w(\a\V)$.
If the pair $(\a\cdot\u,\a\cdot\V)$ is reducible, then by Lemma \ref{Irr} there exist elements $\u',\V'$ and $\y$ such that the pair $(\a\u',\a\V')$ is irreducible and $(\u,\V)=(\u',\V')\y$.
We therefore suppose that the pair $(\a\cdot\u,\a\cdot\V)$ is irreducible and prove by induction on $l(au)+l(av)$ that $(\u,\V) \in \langle H'\rangle$.
If $l(au)+l(av)\leq \m(d(\a),\mathcal{D})$ then certainly $l(au)\leq \m(d(\a),\mathcal{D})$, so by Lemma \ref{Small}, $d(\a\u) \leq 2\m(d(\a),\mathcal{D})$, thus $w(\a\V)\leq w(\a\u) \leq \mathcal{W}'$, so $(\a\u,\a\V) \in H'$.

Suppose now that whenever $(\a\u',\a\V') \in \tau$ is any irreducible pair such that $l(au')+l(av')\leq M$ for some $M\geq \m(d(\a),\mathcal{D})$, then $(\a\u',\a\V') \in \langle H'\rangle$.
Let $(\a\cdot\u,\a\cdot\V) \in \tau$ be an irreducible pair such that $l(au)+l(av) = M+1$.
We are going to show that $(\a\u,\a\V) \in \langle H'\rangle$.
If $w(\a\u) \leq (\mathcal{K}+3)\mathcal{W}'$, then by definition $(\a\u,\a\V) \in H'$, so we can suppose that $w(\a\u)>(\mathcal{K}+3)\mathcal{W}'$.
Of course, this implies that $d(\a\u)>2\text{max}(d(\a),\mathcal{D})$.

Now let
\[
\a\u=\c_1\t_1,\ldots,\d_n\t_n=\a\V
\]
be an irreducible $H$-sequence connecting $\a \u$ and $\a\V$.
Note that $n\geq 1$, for otherwise $\a\u=\a\V$ is an irreducible $H$-sequence such that $d(\a\u)>2\max(d(\a),\mathcal{D})$, which contradicts Lemma \ref{Two}.
By Lemma \ref{Why} we have that there exist elements $\y_1,\ldots,\y_m$ satisfying Conditions (C1)-(C4).
Of course, $\mathcal{W}<\mathcal{W}'$, for the latter corresponds to a doubled diameter.
Furthermore, since $w(\a),w(\y_i)\leq \mathcal{W}'$ for every $i$, we have that $w(\a\y_1\ldots\y_m)\leq (m+1)\mathcal{W}'$.
However, $w(\a\y_1\ldots \y_m) > (\mathcal{K}+3)\mathcal{W}'$, so that making use of 
Lemma~\ref{Basic}, we see that $m> \mathcal{K}+2$.
By Condition (C4), $(\a\y_1 \ldots \y_i)\rho \in \mathbb{K}$ for all $1 \leq i \leq m-1$, so we have that there exist $1\leq i<j\leq \mathcal{K}+1$ such that
\[
\a\y_1\ldots\y_i \mathrel{\rho} \a\y_1\ldots \y_j.
\]
Note that $w(\a\y_1\ldots \y_i),w(\a\y_1\ldots \y_j) \leq (\mathcal{K}+2) \mathcal{W}'$, so we have that the pair
\begin{equation}\label{Nyekk}
(\y_1\ldots\y_i,\y_1\ldots \y_j)
\end{equation}
is contained in $H'$.
For brevity, denote the product $\y_1\ldots \y_i \y_{j+1}\ldots \y_m$ by $\t$.
If we multiply the pair (\ref{Nyekk}) by $\y_{j+1}\ldots\y_m$, we conclude that
\[
(\t,\u) \in \langle H'\rangle,
\]
so $\a\t \mathrel{\rho} \a\V$.
Note that $l(at) < l(au)$, because $\t$ lacks at least one non-idempotent factor (namely $\y_j$).
As a consequence $l(at)+l(av) < l(au)+l(av)=M+1$, so by the induction hypotheses we have that
\[
(\t,\V) \in \langle H'\rangle.
\]
That is, $(\t,\u),(\t,\V) \in \langle H' \rangle$, so by transitivity we have that $(\u,\V) \in \langle H'\rangle$, and the lemma is proved.

\end{proof}

\begin{Lem}\label{B}
Let $\a,\b \in \fla$, $H\subseteq \fla \times \fla$ be finite and let $\rho=\langle H\rangle$ be a finitely generated right congruence.
Then
\[
\a\rho \cdot S \cap \b\rho \cdot S=\{\c\rho: \c \mathrel{\rho} \a\u \mathrel{\rho} \b\V \mbox{ for some }\u,\V \in \fla\}
\]
is either empty or finitely generated as a right $S$-act.
\end{Lem}

\begin{proof}
Suppose that $\a\rho\cdot S \cap \b \rho\cdot S \neq \emptyset$.
Let
\[
\mathbb{K}'= \{\a\u\rho:\text{ there exists }\V \in \fla, \text{ such that } (\a\u,\b\V) \text{ is irreducible}\}.
\]
Note that similarly to the set $\mathbb{K}$ defined before Lemma \ref{K}, $\mathbb{K}'$ is also finite, because by Lemma \ref{Sm}, if $(\a\u,\b\V)$ is irreducible then $\a\u$ is $\rho$-related to an element of $\fla$ having diameter less than or equal to $\m(d(\a),d(\b),\mathcal{D})$.
We claim that $\mathbb{K}'$ generates $\a\rho\cdot S \cap \b\rho\cdot S$.
Let $\a\u\rho=\b\V\rho \in \a\rho\cdot S \cap \b\rho\cdot S$.
Then there exists an $H$-sequence
\[
\a\u=\c_1\t_1,\ldots,\d_n\t_n=\b\V
\]
connecting $\a\u$ and $\b\V$.
By Lemma \ref{Irr}, there exist an irreducible pair $(\a\u',\b\V')$ and $\y \in \fla$ such that $(\a\u,\b\V)=(\a\u',\b\V')\y$.
In this case $\a\u'\rho \in \mathbb{K}'$, so $\a\u\rho \in \mathbb{K}' S$, thus $\mathbb{K}'$ generates $\a\rho\cdot S\cap \b\rho \cdot S$.
\end{proof}

As a consequence of Lemmas \ref{A} and \ref{B} we have our first main result.

\begin{Thm}\label{thm:main1}
If $\Omega$ is finite, then the free left ample monoid $\fla$ is right coherent.
\end{Thm}

To show Theorem~\ref{thm:main1}  is true for arbitrary $\Omega$ we need a simple consequence of Lemma \ref{Crack}.

\begin{Lem} \label{Pi}
Let $\d\z=\b\V$ and let $\Pi$ be a subset of $\Omega$ containing all letters appearing in $D$ and $B$.
Then there exists $\z',\V' \in \mathrm{FLA}(\Pi)$ and $\x \in \FLA$ such that $\d\z'=\b\V'$ and $(\z,\V)=(\z',\V')\x$.
\end{Lem}

\begin{proof}
Let $\z',\V'$ be minimal (with respect to $w(\z')+w(\V')$) in $\FLA$ satisfying that there exists $\x \in \FLA$ such that $\d\z'=\b\V', \z=\z'\x$ and $\V=\V'\x$.
We claim that $\z',\V' \in \mathrm{FLA}(\Pi)$.
Suppose on the contrary that either $\z' \not \in \mathrm{FLA}(\Pi)$ or $\V' \not \in \mathrm{FLA}(\Pi)$.
We can suppose without loss of generality that $\z' \not\in \mathrm{FLA}(\Pi)$.
Then there exists a leaf $x \in Z'$ such that $x$ contains a letter which is not in $\Pi$.
In this case clearly $dx \not \in D \cup B$, so Lemma \ref{Crack} implies that there exist elements $\z'',\V'',\x'$ such that $\d\z''=\b\V'', \z'=\z''\x', \V'=\V''\x'$ and $w(\z'')<w(\z')$.
However, these facts together with the observations $\z=\z'' (\x'\x), \V=\V'' (\x'\x)$ contradict the minimality of $\z'$ and $\V'$.
This shows that $\z',\V' \in \mathrm{FLA}(\Pi)$, finishing the proof.
\end{proof}

\begin{Thm} For any set $\Omega$, we have that $\FLA$ is right coherent.
\end{Thm}
\begin{proof} Let $\rho$ be a right congruence on $\FLA$ with finite set of generators
$H$, so that
$\rho=\langle H\rangle_{\mathrm{FLA}(\Omega)}$, and let $\b,\c\in \FLA$.  
 Let $\Pi$ be the finite set of letters occuring in
$\b,\c$ or in components of $H$ and 
put $\rho'=\langle H\rangle_{\mathrm{FLA}(\Pi)}$.  

We claim that for any  $\u,\V\in \FLA$ with 
$\b\u\,\rho\, \c\V$ via an $H$-sequence
\[\b\u=\c_1\t_1, \d_1\t_1=\c_2\t_2,\hdots, \d_n\t_n=\c\V\]
in $\FLA$,  there exist
\[\u', \t_i'\, (1\leq i\leq n), \V'\in {\mathrm{FLA}(\Pi)}, \x\in \FLA\]
such that
\[\u=\u'\x, \t_i=\t_i'\x\, (1\leq i\leq n), \V=\V'\x\]
and
\[\b\u'=\c_1\t_1', \d_1\t_1'=\c_2\t_2',\hdots, \d_n\t_n'=\c\V'.\]

If $n=0$, then $\b\u=\c\V$ so by Lemma \ref{Pi} we have that $(\u,\V)=(\u',\V')\x$ and $\b\u'=\c\V'$ for some $\u',\V' \in \mathrm{FLA}(\Pi)$ and $\x\in\FLA$ as required.

Suppose now that $n>0$ and the result holds for all sequences of length
$n-1$. Consider the $H$-sequence
\[\b\u=\c_1\t_1, \d_1\t_1=\c_2\t_2,\hdots, \d_n\t_n=\c\V.\]
From the first equality, and the fact that $\c_1\in {\mathrm{FLA}(\Pi)}$, we deduce that
there exists $\u',\t_1'\in {\mathrm{FLA}(\Pi)}$ and $\x\in\FLA$ such that
\[\u=\u'\x, \t_1=\t_1'\x\mbox{ and }\b\u'=\c_1\t_1'.\] 
From the remaining part of the sequence, the fact that $\d_1\in {\mathrm{FLA}(\Pi)}$ and our inductive hypothesis, we deduce that 
there exists $\V'', \t_i''\, (1\leq i\leq n)\in {\mathrm{FLA}(\Pi)}$ and $\z\in\FLA$ such that
\[ \t_i=\t_i''\z, \V=\V''\z\mbox{ and }
\d_1\t_1''=\c_2\t_2'',\hdots, \d_n\t_n''=\c\V''.\]
We now examine the equality
\[\t_1=\t_1'\x=\t_1''\z.\]
Again by Lemma~\ref{Pi} we have that $(\x,\z)=(\x',\z')\w$ for some
$\x',\z'\in {\mathrm{FLA}(\Pi)}$ and $\w\in \FLA$ with $\t_1'\x'=\t_1''\z'$.
Now let
\[\tilde{\u}=\u'\x', \tilde{\t}_i=\t_i''\z' \, (1\leq i\leq n)\mbox{ and }\tilde{\V}=
\V''\z'.\]
Then it is easy to check that

\[\u=\tilde{\u}\w, \t_i=\tilde{\t}_i\w\, (1\leq i\leq n), \V=\tilde{\V}\w\]
and
\[\b\tilde{\u}=\c_1\tilde{\t}_1, \d_1\tilde{\t}_1=\c_2\tilde{\t}_2,\hdots, \d_n\tilde{\t}_n=\c\tilde{\V}.\]
Hence our claim holds by induction.

Since  ${\mathrm{FLA}(\Pi)}$ is right coherent, the right
congruence $r(\a\rho')$ on ${\mathrm{FLA}(\Pi)}$ has a finite set of generators
$K$. Clearly $K\subseteq r(\a\rho)$. Conversely, if $(\u,\V)\in r(\a\rho)$, 
then as $\a\u$ is connected to $\a\V$ via an $H$-sequence, we can apply the above claim to obtain that $\a\u'\,\rho'\, \a\V'$ for some $\u',\V'\in {\mathrm{FLA}(\Pi)}$ such that
$(\u,\V)=(\u',\V')\x$ for some $\x\in \FLA$. Thus
$(\u',\V')\in \langle K\rangle_{{\mathrm{FLA}(\Pi)}}
\subseteq \langle K\rangle_{{\mathrm{FLA}(\Omega)}}$,  and it follows  that
$\langle K\rangle_{{\mathrm{FLA}(\Omega)}}=r(\a\rho)$. 

Now take $\b=\a$ and $\c=\a'$ and suppose  that $\a\rho\cdot \FLA\cap \a'\rho\cdot \FLA\neq\emptyset$. Then
$\a\u\,\rho\, \a'\V$ for some $\u,\V\in \FLA$ and we have that $\a\u'\,\rho'\, \a'\V'$ for some
 $\u',\V'\in {\mathrm{FLA}(\Pi)}$ such that
$(\u,\V)=(\u',\V')\x$ for some $\x\in \FLA$. Since
$\a\rho'\cdot {\mathrm{FLA}(\Pi)}\cap \a'\rho'\cdot {\mathrm{FLA}(\Pi)}\neq\emptyset$ and ${\mathrm{FLA}(\Pi)}$ is right coherent, we have that
$\a\rho'\cdot {\mathrm{FLA}(\Pi)}\cap \a'\rho'\cdot {\mathrm{FLA}(\Pi)}= L\cdot {\mathrm{FLA}(\Pi)}$ for some finite set $L=\{ \u_i\rho': 1\leq i\leq n\}$,
where the $\u_i$ are fixed representatives of their $\rho'$-classes. 

For each $i\in \{ 1,\hdots,n\}$ we therefore have that
\[\a\w_i\,\rho'\, \u_i\x_i\,\rho'\, \a'\z_i\]
for some $\w_i,\x_i,\z_i\in {\mathrm{FLA}(\Pi)}$, so that clearly
\[\a\w_i\,\rho\, \u_i\x_i\,\rho\, \a'\z_i\]
and so 
\[L'=\{ \u_i\rho: 1\leq i\leq n\}\subseteq \a\rho\cdot \FLA\cap \a'\rho\cdot \FLA.\]

Conversely, if $\a\b\,\rho\, \a'\c$ then as above we have that
$(\b,\c)=(\b',\c')\t$ for some $\b',\c'\in {\mathrm{FLA}(\Pi)}$ and $\t\in \FLA$ with
$\a\b'\,\rho'\,\a'\c'$. Now $(\a\b')\rho' =(\u_i\rho')\w$ for some 
$i\in \{ 1,\hdots, n\}$ and $\w\in {\mathrm{FLA}(\Pi)}$ so that
$(\a\b')\rho =(\u_i\rho)\w$ and hence $(\a\b)\rho=(\u_i\rho)\w\t\in L'\cdot \FLA$. Thus 
$\a\rho\cdot \FLA\cap \a'\rho\cdot \FLA =L'\cdot \FLA$ as required.
\end{proof}

\section{Coherency and retracts}\label{sec:constructions}

Investigations of how coherency behaves with respect to certain constructions will be the subject of a future paper, however, to show how the coherency of the free monoid follows from our result, we show that retracts of (right) coherent monoids are (right) coherent.

\begin{Def} Let $S$ be a monoid. Then $T\subseteq S$ is a {\em retract} of $S$ if there exists a homomorphism $\phi \colon S \to S$ such that $\phi^2=\phi$ and $\text{Im }\phi=T$.

Note that any retract is a subsemigroup and a monoid.
\end{Def}

\begin{Lem}
Let $S$ be a monoid and let $T$ be a retract of $S$.  
Let $\rho$ be a right congruence on $T'$, and let $\rho'$ be the right congruence on $S$ generated by $\rho$.
Then the restriction of $\rho'$ to $T$ coincides with $\rho$.
\end{Lem}

\begin{proof}
Let $a,b \in T$ such that $a \mathrel{\rho'} b$.
Since $\rho'$ is generated by $\rho$, there exist elements $c_1,\ldots,c_n,d_1,\ldots,d_n \in T$ and $t_1,\ldots, t_n \in S$ such that $c_i\mathrel{\rho} d_i$ for every $1\leq i\leq n$, and such that
\[
a=c_1t_1,\ldots, d_nt_n=b.
\]
If we take the image of this sequence under $\phi$ we obtain the $H$-sequence
\[
a=c_1(t_1\phi),\ldots,d_n(t_n\phi)=b
\]
connecting $a$ and $b$ in $T$, so $a \mathrel{\rho} b$.

\end{proof}

\begin{Thm}\label{Retract}
Let $S$ be a right coherent monoid and let $T$ be a retract of $S$.
Then $T$ is right coherent.
\end{Thm}

\begin{proof}
Let $\rho$ be a finitely generated right congruence on $T$, 
so that $\rho=\langle H\rangle_T$ for some finite set $H \subseteq T \times T$.
Denote by $\rho'$ the right congruence on $S$ generated by $\rho$.
Clearly, $\rho'=\langle H\rangle_S$.

First we show that if $a,b \in S$ and $a \mathrel{\rho'} b$, then $a\phi \mathrel{\rho} b\phi$.
For this, let
\[
a=c_1t_1,\ldots, d_nt_n=b
\]
be an $H$-sequence connecting $a$ and $b$ in $S$.
Since $H \subseteq T \times T$, if we take the image of this sequence under $\phi$ we obtain the $H$-sequence
\[
a\phi=c_1(t_1\phi),\ldots,d_n(t_n\phi)=b\phi
\]
connecting $a\phi$ and $b\phi$ in $T$, so that $a\phi \mathrel{\rho} b\phi$.

Now let $a \in T$ be fixed.
Note that $r(a\rho')$ is a right congruence on $S$, and $r(a\rho)$ is a right congruence on $T$.
Since $S$ is right coherent, we have that $r(a\rho')=\langle X\rangle_S$ for some finite $X \subseteq S \times S$.
We claim that the finite set
\[
X\phi=\{(u\phi,v\phi): (u,v) \in X\}\subseteq T \times T
\]
generates $r(a\rho)$.

First note that if $(u,v) \in X$, then $au \mathrel{\rho'} av$, so we have that \[a(u\phi)=(au)\phi \mathrel{\rho} (av)\phi=a(v\phi),\] that is, $(u\phi,v\phi) \in r(a\rho)$.
Thus we have shown that $X\phi \subseteq r(a\rho)$.

On the other hand, if $(u,v) \in r(a\rho)$, then necessarily $(u,v) \in r(a\rho')$, so there exists an $X$-sequence
\[
u=c_1t_1,\ldots,d_nt_n=v
\]
connecting $u$ and $v$ in $S$.
If we take the image of this sequence under $\phi$ (and remember that $u,v \in T$), then we obtain the $X\phi$-sequence
\[
u=(c_1\phi) (t_1\phi),\ldots,(d_n\phi)(t_n\phi)=v
\]
connecting $u$ and $v$.
That is, $(u,v) \in \langle X\phi \rangle_T$, and we have shown that $r(a\rho)$ is finitely generated.

Now suppose that $a,b \in T$ are such that $a\rho\cdot T \cap b\rho\cdot T \neq \emptyset$.
Then clearly $a\rho'\cdot S \cap b\rho'\cdot S \neq \emptyset$, so there exists a finite set $Y \subseteq S$ such that $a\rho'\cdot S \cap b\rho'\cdot S=Y\cdot S$.
We claim that $a\rho\cdot T \cap b\rho\cdot T = Y\phi\cdot T$ where
\[
Y\phi=\{(x\phi) \rho: x\rho' \in Y\}\subseteq T \times T.
\]
Notice that $Y\phi$ is well defined, for if $x \mathrel{\rho'} y$, then $x\phi \mathrel{\rho} y\phi$.

First note that if $x\rho' \in Y$, then $au \mathrel{\rho'} x \mathrel{\rho'} bv$ for some $u,v \in S$.
By an earlier comment,  this implies that $a(u\phi) \mathrel{\rho} x\phi \mathrel{\rho} b(v\phi)$, so $(x\phi)\rho \in a\rho\cdot T \cap b\rho\cdot T$, and so $Y \phi \cdot  T \subseteq a\rho\cdot T \cap b\rho\cdot T$.

Conversely, let $w\rho \in a\rho\cdot T \cap b\rho\cdot T$ for some $w \in T$.
Then clearly $w\rho' \in a\rho'\cdot S \cap b\rho'\cdot S$, so there exist an $x\rho' \in Y$ and $s\in S$ such that $w\rho'=x\rho' \cdot s$, that is, $w \mathrel{\rho'} xs$.
Applying $\phi$ we see that $w=w\phi \mathrel{\rho} (x\phi) (s\phi)$, that is, $w\rho=(x\phi)\rho \cdot s\phi\in
 Y\phi\cdot T$. Consequently,  $a\rho\cdot T \cap b\rho\cdot T \subseteq Y\phi\cdot T$ as required.
\end{proof}

\begin{Cor} \cite{ghr:2013}
The free monoid $\Omega^*$ is right coherent.
\end{Cor}

\begin{proof}
Note that the idempotent map
\[
\phi \colon \fla \to \Omega^*, \a \mapsto (a\da,a)
\]
is a homomorphism, so $\Omega^*$ is a retract of $\fla$.
Then Theorem \ref{Retract} implies that $\Omega^*$ is right coherent.
\end{proof}

Note that the free monoid is (right) coherent, however, there exist non-coherent monoids, so the class of (right) coherent monoids is not closed under homomorphic images.

\section{The negative results}\label{sec:negative}

In this section, we show that the free inverse monoid is not left coherent. By duality, neither can it be right coherent.
A few simple remarks then yield that the free left ample monoid is not left coherent and that the free ample monoid is neither left nor right coherent.

Let $\Omega=\{x,y\},\ \a=(\{\epsilon,x\},x) \in \FI$ and $\b=(\{\epsilon,y\},y) \in \FI$.
Denote by $\rho$ the left congruence generated by the pair $(\a,\bf{1})$, and by $\tau$ the left annihilator of $\b\rho$, that is,
\[
\tau=\{(\u,\V): \u\b \mathrel{\rho} \V\b\} \subseteq \FI \times \FI.
\]
It is easy to see that $\tau$ is a left congruence on $\FI$.
We claim that it is not finitely generated. 

The following lemma is effectively folklore, but we prove it here for completeness.

\begin{Lem}\label{Rho}
For every $\u,\V \in \FI$, we have that $\u \mathrel{\rho} \V$ if and only if there exist $m,n \in \mathbb{N}^0$ such that $\u\a^n=\V\a^m$.
\end{Lem}

\begin{proof}
It is straightforward that if such $n$ and $m$ exist, then $\u$ and $\V$ are $\rho$-related. For the converse part, suppose that $\u \mathrel{\rho} \V$. Thus, since $\rho$ is generated by $(\a,\bf{1})$, there exist elements $\c_1,\ldots,\c_p,\d_1,\ldots,\d_p,\t_1,\ldots,\t_p \in \FI$ such that for any $1\leq i\leq p$, $(\c_i,\d_i)=(\a,\bf{1})$ or $(\c_i,\d_i)=(\bf{1},\a)$, satisfying
\[
\u=\t_1\c_1,\t_1\d_1=\t_2\c_2,\ldots,\t_{p-1}\d_{p-1}=\t_p\c_p,\t_p\d_p=\V.
\]

Note that for all $1\leq i\leq p$, we have that either $\t_i\c_i=\t_i\d_i\a$ (exactly when $(\c_i,\d_i)=(\a,{\bf 1})$) or $\t_i\c_i\a=\t_i\d_i$ (exactly when $(\c_i,\d_i)=({\bf 1},\a)$). Applying this argument successively to $i=1,2,\ldots,p$, we obtain the result of the lemma (actually, we also see that $n$ and $m$ are just the number of the pairs $(\bf{1},\a)$ and $(\a,\bf{1})$, respectively, in the sequence $(\c_1,\d_1),\ldots,(\c_p,\t_p)$).
\end{proof}

As a direct consequence, we have the following lemma:

\begin{Lem}\label{Tau}
For every $\u,\V \in \FI$, $\u\mathrel{\tau}\V$ if and only if there exist $m,n \in \mathbb{N}^0$ such that $\u\b\a^n=\V\b\a^m$.
\end{Lem}

For any $0 \leq i$, let 
\[
U_i=\{\epsilon,y,yx,\ldots, yx^i\}.
\]

\begin{Lem} \label{prevv}
We have that $(U_i,\epsilon) \mathrel{\tau} (U_1,\epsilon)$ for any $1\leq i$.
\end{Lem}

\begin{proof}
Since
\[
(U_i,\epsilon)\b\a^{i}=(U_1,\epsilon)\b\a^{i}=(\{\epsilon,y,yx,yx^2,\ldots,yx^{i})
\]
we have by Lemma \ref{Tau} that $(U_i,\epsilon) \mathrel{\tau} (U_1,\epsilon)$.
\end{proof}

\begin{Lem}\label{lem:nfg}
The left annihilator congruence $\tau=l(\b\rho)$ is not finitely generated.
\end{Lem}

\begin{proof} Suppose for contradiction that
 $H$ is a finite symmetric subset of $\tau$ generating $\tau$ and let $k$ be a natural number such that for every $((S,s),(T,t)) \in H$ we have that $k > \left|S\right|$.

Now suppose that $(U_k,\epsilon)=\t\c$ where $(\c,\d) \in H$ and $\t \in \FI$.
Then $c^{-1}=t \in U_k$ and $c^{-1}C \subseteq U_k$.
Note that since $c \in C$, $c^{-1}C$ is also prefix closed.
The facts that $U_k$ is a single path and $\left|C\right|<k$ imply that $c^{-1}C \subseteq \{\epsilon,y,yx,\ldots,yx^{k-1}\}$.
However, $U_k=T \cup c^{-1}C$, and as a consequence we have that $yx^k \in T$, so $T=U_k$.

We also have $\c \mathrel{\tau} \d$, so there exist $i,j$ such that $\c\b\a^i=\d\b\a^j$.
By multiplying this equality from the right by an appropriate power of $\a$ we can ensure that $i,j>k$.
Note that since $C \subseteq cU_k$, the first component of $\c\b\a^i$ is $\{c,cy,cyx,\ldots,cyx^i\}$, whereas the first component of $\d\b\a^j$ contains the vertices $\{d,dy,dyx,dyx^2,\ldots,dyx^j\}$.
Given that $c^{-1}\in U_k$, a brief analysis shows this can only happen if $d=c$, and then $c^{-1}D \subseteq \{\epsilon,y,\ldots, yx^{k-1}\}$ follows from the facts that
\[
c^{-1}D \subseteq c^{-1} \{c,cy,cyx,\ldots, cyx^{i}\}=\{\epsilon,y,yx,\ldots, yx^{i}\},
\]
$c^{-1}D$ is prefix closed and $\left|c^{-1}D\right|<k$.
So altogether we obtain that $T=U_k$ and $tD=c^{-1}D \subseteq \{\epsilon,y,yx,\ldots,yx^{k-1}\} \subseteq U_k$, so $T \cup tD=U_k$ and as a consequence we conclude that $\t\d=(U_k,\epsilon)$.
That is, applying elements of $H$ to right factors of $(U_k,\epsilon)$ does not change $(U_k,\epsilon)$, so the $\tau$-class of $(U_k,\epsilon)$ is singleton, that is, $(U_k,\epsilon) \not\mathrel{\tau} (U_{k+1},\epsilon)$, contradicting Lemma \ref{prevv}.
\end{proof}

\begin{Thm}
Let $\left|\Omega\right|>1$.
Then the free inverse monoid $\FI$ and the free ample monoid $\FA$ are neither left nor right coherent.
The free left ample monoid $\FLA$ is right coherent, but not left coherent.
\end{Thm}

\begin{proof}
Lemma \ref{lem:nfg} shows that $\FI$ is not left coherent. Exactly the same argument applies to show that $\FLA$ and $\FA$
are not left coherent, simplifying further, since $c=t=\epsilon$. By duality, $\FI$ and $\FA$ cannot be right coherent. 
\end{proof}


\begin{thebibliography}{99}


\bibitem{chase:1960} S. U. Chase,`Direct product of modules', {\em Trans. Amer. Math.
Soc.} {\bf 97} (1960), 457--473.

\bibitem{choo:1972} K. G. Choo, K. Y. Lam, and E. Luft,
`On free product of rings and the coherence property',  
{\em Algebraic K-theory, II:``Classical'' algebraic K-theory and connections with arithmetic}
(Proc. Conf., Battelle Memorial Inst., Seattle, Wash., 1972), pp. 135--143,  
{\em Lecture Notes in Math.} {\bf 342}, Springer, Berlin, 1973. 

\bibitem{cohn:1971} P. M. Cohn, `Free rings and their relations',  Academic Press, London \&
New York, 1971. 


\bibitem{cornock:2011} C. Cornock, {\em Restriction Semigroups:  Structure, Varieties and Presentations} PhD thesis,
University of York, 2011.

\bibitem{eklof:1971} P. Eklof and G. Sabbagh, `Model-completions and modules', {\em Annals of Math. Logic} {\bf 2}
(1971), 251--295.

\bibitem{fountain:1991} J. B. Fountain, `Free right type A semigroups',
{\em Glasgow Math. J.} {\bf 33} (1991), 135--148.


\bibitem{gould:2009} J. B. Fountain, G. M. S. Gomes and  V. Gould, `The free ample monoid', {\em  I.J.A.C.} {\bf 19} (2009), 527-554. 

\bibitem{gomes:2000} G. M. S. Gomes and V. Gould, `Graph expansions of unipotent
monoids', {\em Communications in Algebra} {\bf 28} (2000), 447--463.

\bibitem{gould:1987} V. Gould, `Model Companions of S-systems', {\em Quart. J. Math. Oxford} {\bf 38} (1987), 189--211. 

\bibitem{gould:1987b} V. Gould, `Axiomatisability problems for $S$-systems', {\em J. London Math. Soc.} {\bf 35} (1987), 193--201.

\bibitem{gould:1992} V. Gould, `Coherent monoids', {\em J. Australian Math. Soc.} {\bf 53} (1992), 166--182. 

\bibitem{ghr:2013} V. Gould, M. Hartmann and N. Ru\v{s}kuc, `The free monoid is coherent'  preprint at {{\em http://www-users.york.ac.uk/$\sim$varg1/gpubs.htm}}. 

\bibitem{ho}{J.\ M.\ Howie}, {\it Fundamentals of semigroup theory},
Oxford University Press, 1995.
\bibitem{lawsonbook} M. V. Lawson, {\em Inverse semigroups: The Theory of
Partial Symmetries}, World Scientific 1998.
	

\bibitem{munn:1974} W. D. Munn `Free inverse semigroups', {\em Proc. London
Math.
Soc.} {\bf 29} (1974), 385--404.


\bibitem{scheiblich:1972} H. E. Scheiblich, `Free inverse semigroups',
{\em Semigroup Forum} {\bf 4} (1972), 351--358.

\bibitem{wheeler:1976} W. H. Wheeler, `Model companions and definability in 
existentially complete structures', {\em Israel J. Math.} {\bf 25} (1976),
305--330.
\end{thebibliography}
\end{document}